\documentclass[11pt]{amsart}

\usepackage{amsmath,amsfonts,amsthm,amssymb,mathrsfs}
\usepackage[utf8]{inputenc}
\usepackage{graphicx}
\usepackage[english]{babel}
\usepackage[margin=3cm]{geometry} 
\usepackage{mathtools}
\usepackage[colorlinks=true, allcolors=blue]{hyperref}
\usepackage{calligra}
\usepackage{tikz-cd,adjustbox}
\usepackage{spverbatim,comment}

\usepackage[alphabetic,initials]{amsrefs}

\counterwithin{equation}{subsection}
\counterwithout{subsection}{section}
\counterwithin{figure}{subsection}

\newtheorem{theorem}[equation]{Theorem}
\newtheorem*{theorem*}{Theorem}
\newtheorem{lemma}[equation]{Lemma}
\newtheorem*{lemma*}{Lemma}
\newtheorem{corollary}[equation]{Corollary}
\newtheorem{proposition}[equation]{Proposition}
\newtheorem*{proposition*}{Proposition}

\theoremstyle{definition}
\newtheorem{definition}[equation]{Definition}
\newtheorem*{definition*}{Definition}
\newtheorem{remark}[equation]{Remark}

\newtheorem{example}[equation]{Example}
\newtheorem*{example*}{Example}
\newtheorem*{problem*}{Problem}

\theoremstyle{plain}

\newcommand{\D}{\mathcal D}
\newcommand{\E}{\mathcal E}
\newcommand{\F}{\mathcal F}

\renewcommand{\H}{\mathcal H}

\newcommand{\M}{\mathcal M}

\renewcommand{\O}{\mathcal O}
\renewcommand{\P}{\mathcal P}

\newcommand{\bD}{\mathbf D}

\newcommand{\CC}{\mathbb C}

\newcommand{\HH}{\mathbb H}

\newcommand{\LL}{\mathbb L}
\newcommand{\NN}{\mathbb N}
\newcommand{\PP}{\mathbb P}
\newcommand{\QQ}{\mathbb Q}

\newcommand{\ZZ}{\mathbb Z}
\newcommand{\DD}{\mathbb D}

\newcommand{\sD}{\mathscr D}

\newcommand{\sS}{\mathscr S}
\newcommand{\sT}{\mathscr T}

\newcommand{\xto}{\xrightarrow} 
\DeclareMathOperator{\Hom}{Hom}
\newcommand{\sHom}{\H om} 

\newcommand{\sExt}{\E xt} 

\newcommand{\DB}{\underline{\Omega}} 

\DeclareMathOperator{\DR}{DR}

\DeclareMathOperator{\sing}{sing}
\DeclareMathOperator{\supp}{supp}
\DeclareMathOperator{\Spec}{Spec}

\DeclareMathOperator{\Pic}{Pic}

\DeclareMathOperator{\lcd}{lcd}
\DeclareMathOperator{\codim}{codim}

\DeclareMathOperator{\IC}{IC}
\DeclareMathOperator{\reg}{reg}

\DeclareMathOperator{\gr}{{\rm gr}}

\DeclareMathOperator{\Alb}{Alb}
\DeclareMathOperator{\Perv}{Perv}
\DeclareMathOperator{\red}{red}
\DeclareMathOperator{\Dp}{\prescript{p}{}{D}}
\DeclareMathOperator{\lcdef}{lcdef}

\newcommand{\Kahler}{K\"{a}hler}
\newcommand{\Poincare}{Poincar\'{e}}
\newcommand{\Kollar}{Koll\'ar}
\newcommand{\Kovacs}{Kov\'acs}

\DeclarePairedDelimiter\abs{\lvert}{\rvert}

\newcommand{\quotes}[1]{``#1"}

\makeatletter
\let\oldabs\abs
\def\abs{\@ifstar{\oldabs}{\oldabs*}}
\makeatother

\newcommand{\theoremref}[1]{\hyperref[#1]{Theorem~\ref*{#1}}}
\newcommand{\lemmaref}[1]{\hyperref[#1]{Lemma~\ref*{#1}}}
\newcommand{\definitionref}[1]{\hyperref[#1]{Definition~\ref*{#1}}}
\newcommand{\propositionref}[1]{\hyperref[#1]{Proposition~\ref*{#1}}}
\newcommand{\conjectureref}[1]{\hyperref[#1]{Conjecture~\ref*{#1}}}
\newcommand{\corollaryref}[1]{\hyperref[#1]{Corollary~\ref*{#1}}}
\newcommand{\exampleref}[1]{\hyperref[#1]{Example~\ref*{#1}}}

\makeatletter
\let\old@caption\caption
\renewcommand*{\caption}[1]{%
	\setcounter{figure}{\value{equation}}%
	\stepcounter{equation}%
	\old@caption{#1}\relax%
}
\makeatother

\newcounter{intro}

\newtheorem{intro-conjecture}[intro]{Conjecture}
\newtheorem{intro-corollary}[intro]{Corollary}
\newtheorem{intro-theorem}[intro]{Theorem}
\newtheorem{intro-proposition}[intro]{Proposition}

\begin{document}

\title{Generic vanishing for singular varieties via Du Bois complexes}
\author[A. D.~Vo]{Anh~Duc~Vo}
\address{Department of Mathematics, Harvard University, 
1 Oxford Street, Cambridge, MA 02138, USA} 
\email{{\tt ducvo@math.harvard.edu}}

\thanks{}
\date{\today}

\subjclass[2020]{}

\begin{abstract}
We prove appropriate generic vanishing theorems for singular varieties, generalizing the well-known generic vanishing theorem by Green and Lazarsfeld in \cite{GL87} and the generic vanishing theorem of Nakano type in \cite{PS_generic_vanishing}. Our theorem explains the counterexample of Hacon and \Kovacs{} in \cite{Hacon-counter_eg_for_GV}. 
\end{abstract}

\maketitle

\makeatletter
\newcommand\@dotsep{4.5}
\def\@tocline#1#2#3#4#5#6#7{\relax
  \ifnum #1>\c@tocdepth 
  \else
    \par \addpenalty\@secpenalty\addvspace{#2}%
    \begingroup \hyphenpenalty\@M
    \@ifempty{#4}{%
      \@tempdima\csname r@tocindent\number#1\endcsname\relax
    }{%
      \@tempdima#4\relax
    }%
    \parindent\z@ \leftskip#3\relax
    \advance\leftskip\@tempdima\relax
    \rightskip\@pnumwidth plus1em \parfillskip-\@pnumwidth
    #5\leavevmode\hskip-\@tempdima #6\relax
    \leaders\hbox{$\m@th
      \mkern \@dotsep mu\hbox{.}\mkern \@dotsep mu$}\hfill
    \hbox to\@pnumwidth{\@tocpagenum{#7}}\par
    \nobreak
    \endgroup
  \fi}
\def\l@section{\@tocline{1}{0pt}{1pc}{}{\bfseries}}
\def\l@subsection{\@tocline{2}{0pt}{25pt}{5pc}{}}
\makeatother


\tableofcontents

\section{Introduction}

Cohomology vanishing statements play an essential role in algebraic geometry, starting with the well-known Kodaira vanishing theorem and extending to more general ones like Kawamata-Viehweg vanishing, \Kollar{} vanishing, or Saito vanishing theorems. All these vanishing theorems depend on positivity assumptions, and the attempt to relax these assumptions leads to the development of generic vanishing theory. One of the most famous results is the pioneering work of Green and Lazarsfeld (see \cite{GL87}*{Theorem 1}), which states that on a smooth projective variety $X$, the locus of topologically trivial line bundles $L$ such that $H^i(X,L) \neq 0$ has codimension at least $\dim a(X)-i$ in $\Pic^0(X)$, where $a:X \to \Alb(X)$ denotes the Albanese map of $X$. One can consider the set of topologically trivial line bundles for which the cohomology in a given degree does not vanish. More precisely, following \cite{PP_Generic_Vanishing}, we denote, for any coherent sheaf $F$ (or more generally $F\in D_{coh}^b(X)$), by 
\[V^i(F)\coloneqq \{L \in \Pic^0(X)\ |\  H^i(X, F\otimes L)\neq 0\}\]
the \textit{i-th cohomological support locus} of $F$, and we say that $F$ is $GV_{-k}$ if 
\[\codim(V^i(F), \Pic^0(X)) \ge i-k\]
for all $i$. In this language, the classical generic vanishing theorem of Green and Lazarsfeld can be stated as:
\begin{equation} \label{classical GV theorem}
    \text{The canonical line bundle } \omega_X \text{ is } GV_{\dim a(X)- \dim X}.
\end{equation} 
This theorem and its variants have found numerous applications, ranging from the study of singularities of theta divisors \cite{EL-sing_of_theta} to the work on birational geometry of irregular varieties \cite{Hacon-Ueno_conj}. Later, by employing the theory of mixed Hodge modules, Popa and Schnell proved generic vanishing results for Hodge theoretic objects on abelian varieties. These led to a generic vanishing theorem of Nakano type \cite{PS_generic_vanishing}*{Theorem 3.2}:
\begin{equation} \label{PS generic Nakano vanishing}
    \text{The bundle of holomorphic $p$-forms } \Omega^p_X \text{ is } GV_{p-\dim X-\delta(a)},
\end{equation}
where $\delta(a)$ is the semismallness defect of the Albanese morphism $a:X\to \Alb(X)$ (see \ref{definition: semismall-defect} for the definition), and led to the study of zeros of holomorphic one forms on smooth complex projective varieties \cite{zeros_of_holo_forms}. 
    
One might ask if generic vanishing theorems, say for the dualizing sheaf $\omega_X$ to begin with, continue to hold for singular varieties. The answer is positive if $X$ has rational singularities, where the situation is essentially identical to that of smooth varieties. On the other hand, Hacon and \Kovacs{} \cite[Theorem 1.3]{Hacon-counter_eg_for_GV} constructed an example in which $X$ is singular and the classical generic vanishing theorem (\ref{classical GV theorem}) fails. Given the proofs of the known generic vanishing theorems, it is natural to expect this failure due to the fact that Hodge theory operates quite differently in the case of singular varieties.  
In this article, inspired by the use of Hodge module theory in \cite{PS_generic_vanishing}, we establish natural generic vanishing statements from the point of view of the Hodge theory of singular varieties, which, in particular, explain that the example in \cite{Hacon-counter_eg_for_GV} still fits well in the general framework. 

Our results are phrased in terms of Du Bois complexes $\DB_X^p$ and intersection Du Bois complexes $I\DB_X^p$ (see Section \ref{section: Du Bois and intersection complex} for the definitions). Given a complex algebraic variety $X$, for each $p$, the \textit{$p$-th Du Bois complex} $\DB_X^p$ is the (shifted) $p$-th associated graded term of the filtered de Rham complex $\DB_X^\bullet$ with respect to the Hodge filtration. Du Bois complexes generalize sheaves of holomorphic $p$-forms $\Omega_X^p$, and play an important role in the study of singularities via Hodge modules. On the other hand, the intersection complex $\IC_X^\bullet$ and its (shifted) associated graded terms $I\DB_X^p$ with respect to the Hodge filtration, which we call \textit{intersection Du Bois complexes}, were introduced to generalize the \Poincare{} duality. In fact, we have: 
\[I\DB_X^p \cong \DD_X(I\DB_X^{d_X-p})[-d_X]\]
where $\DD_X = R\Hom_{\O_X}(-,\omega_X^\bullet)$ is the Grothendieck duality and $d_X$ denotes the dimension of $X$. Due to \cite{Kebekus-Schnell}*{Proposition 8.1}, the intersection Du Bois complexes are related to sheaves of holomorphic forms of a resolution of singularities of $X$ as:
\[\H^0 I\DB_X^p \cong f_* \Omega_{\tilde X}^p\]
where $f: \tilde X \to X$ is a resolution of singularities of $X$ such that $f$ is isomorphic away from $X_{\sing}$ and the (reduced) exceptional locus is simple normal crossing. Moreover, both the top Du Bois complex $\DB_X^{d_X}$ and the top intersection Du Bois complex $I\DB_X^{d_X}$ are isomorphic to $f_* \omega_{\tilde X}$.  
We refer the readers to \cite{Mihnea-Sunggi} for the discussions of Du Bois complexes and intersection Du Bois complexes. To simplify the statements, we denote by $d_X$ the dimension of a variety $X$. Here are the main results:

\medskip

\begin{intro-theorem} \label{theorem: easy GV for DB^p and IC^p}
    Let $a:X\to A$ be a morphism from a projective variety $X$ to an abelian variety $A$. Then
    \begin{enumerate}
        \item $\DB^{d_X}_X \cong I\DB_X^{d_X}$ is $GV_{d_{a(X)}-d_X}$. 
        \item $\DB^{p}_X$ is $GV_{p-d_X -\delta(a)}$ and $I\DB_X^{p}$ is $GV_{p-d_X -\delta_s(a)}$ for any $p\in \NN$.  
    \end{enumerate}
\end{intro-theorem}

Here $\delta(a)$ (resp. $\delta_s(a)$) is the semismallness defect (resp. stratified semismallness defect) of the morphism $a$. They are defined by the formula:
\[\delta(a) \coloneqq \max_{l\in \NN} \{2l-\dim(X) + \dim(A_l)\}\]
where $A_l = \{y\in A | \dim a^{-1}(y)\ge l \}$ and 
\[\delta_s(a) \coloneqq \max_{s\in \sS} \{\delta(a|_{X_s})\}\]
where $X = \bigsqcup_{s\in \sS} X_s$ is the minimal Whitney stratification of $X$. When $\delta(a) =0$ the morphism $a$ is semismall, while $\delta_s(a) = 0$ implies that $a|_{X_s}$ is semismall map for all strata $X_s$. Statement $(1)$ generalizes the classical generic vanishing theorem (\ref{classical GV theorem}), while statement $(2)$ generalizes the generic Nakano vanishing theorem (\ref{PS generic Nakano vanishing}). On the other hand, for singular varieties, the Grothendieck duals $\DD_X(\DB_X^{d_X-k})[-d_X]$ play an important role but are usually different from $\DB^{k}_X$. These objects also give a natural generalization of $\Omega_X^{k}$; in particular, when $k=d_X$, the more appropriate generalization of the dualizing complex $\omega_X$ is $\DD_X(\DB_X^0)$. For instance, we have $\DB_X^{d_X}\cong I\DB_X^{d_X} \cong \omega_X$ if $X$ has rational singularities, while we only need $X$ to be Cohen-Macaulay and have Du Bois singularities to get $\DD_X(\DB^0_X)[-d_X] \cong \omega_X$. It is therefore necessary to have a statement for these objects as well.
    
\begin{intro-theorem} \label{theorem: hard GV for D(DB^p)}
    Let $X$ be an irreducible projective variety with a morphism $a: X \to A$ to an abelian variety. Then 
    \begin{enumerate}
        \item $\DD_X(\DB^0_X)[-d_X]$ is $GV_{d}$ where $d=\min\{d_{a(X)}-d_X, -d_F(a) -\lcdef(X)\}$. 
        \item $\DD_X(\DB_X^{d_X-p})[-d_X]$ is $GV_{p-d_X-\delta_s(a)-\lcdef(X)}$ for any $p\in\NN$.
    \end{enumerate}
\end{intro-theorem}

Here $\lcdef(X)$ is the local cohomological defect of $X$, and $d_F(a)$ is the stratified relative dimension of the morphism $a$ (see Definition \ref{definition: lcdef} and \ref{definition: stratified relative dimension}). The \textit{local cohomological defect} of $X$ is defined as 
\[\lcdef(X) = \lcd(X,Y) - \codim(X,Y)\]
where $X\subseteq Y$ is an embedding in a smooth variety, with local cohomological dimension $\lcd(X,Y)$. It does not depend on the embedding, and there are intersting classes of varieties with $\lcdef(X) = 0$, for instance local complete intersections. The stratifed relative dimension $d_F(a)$ is defined as the maximal relative dimension of the maps $a|_{X_s}$ where $X_s$ is a strata inside the singular locus of $X$. In fact, this invariant explains the failure of the classical generic vanishing theorem for the example in \cite{Hacon-counter_eg_for_GV}. Hacon and \Kovacs{} constructed a generically finite morphism from a Gorenstein log canonical variety $T$ to an abelian variety $A$, which contracted the one-dimensional singular locus of $X$ to a point, and showed that $\DD_X(\DB^0_X)[-d_X] \cong \omega_X$ is not $GV_0$. In this example, we have $d_{a(X)}-d_X = \lcdef(X) = 0$ and $d_F(a) = 1$, and our Theorem \ref{theorem: hard GV for D(DB^p)} indicates that $\omega_X$ should be $GV_{-1}$.

As a consequence of Theorem \ref{theorem: hard GV for D(DB^p)}, we deduce the non-negativity of Euler characteristic.

\begin{intro-proposition} \label{proposition: nonnegativity_Euler_characteristic}
    Let $X$ be normal variety with maximal Albanese dimension such that $d_F(a) = 0$ and $\lcdef(X)=0$, then $\DD_X(\DB_X^0)[-d_X]$ is a $GV$-sheaf and $\chi(\DD_X(\DB_X^0)[-d_X])\ge 0$. 
    
    In particular, if $X$ also has Du Bois singularities, then $X$ is Cohen-Macaulay and $\chi(\omega_X)\ge 0$. 
\end{intro-proposition}

\medskip
\noindent \textbf{Acknowledgement.} I would like to express my sincere gratitude to my advisor \\
Mihnea Popa for his constant support during the preparation
of this paper. I thank Christian Schnell, S\' andor  \Kovacs{}, Sung Gi Park, and Wanchun Shen for valuable conversations, and Bradley Dirks and Benjamin Church for comments on the draft.

\section{Preliminaries}
In this paper, we work over $\CC$. We denote by $D_{coh}^b(X)$ (resp. $D^b_c(X)$) the bounded derived category of coherent sheaves (resp. constructible sheaves over $\CC$) on $X$. The Grothendieck duality of $F\in D_{coh}^b(X)$ is defined as
\[\DD_X(F) \coloneqq R\sHom_{\O_X}(F,\omega_X^\bullet)\]
where $\omega_X^\bullet$ is the dualizing complex of $X$. Note that when $X$ is smooth, we have $\DD_X(F) \cong R\sHom_{\O_X}(F, \omega_X)[d_X]$ which explains the difference of the indices in our results, comparing to the existing literature in the smooth case.

A resolution of singularity $f: \tilde X \to X$ is called a log resolution of $X$ if $f$ is an isomorphism over the regular locus $X_{\reg}$ and $f^{-1}(X_{\sing})_{\red}$ is a simple normal crossing divisor where $X_{\sing}$ is the singular locus of $X$.

\subsection{Mixed Hodge modules}

We only recall some facts about mixed Hodge modules needed in this article. The main references are Saito's papers \cite{Saito-88} and \cite{saito-MHM}.

Given a smooth complex algebraic variety $Y$, we denote by $\sD_Y$ the sheaf of differential operators on $Y$. We refer to \cite{HTT} for a complete treatment of $\sD$-modules. All the $\sD$-modules we consider in this paper are right $\sD$-modules. For any coherent $\sD_Y$-module $\M$, there exists a \textit{good} filtration $F_*\M$ i.e. an increasing, exhaustive, bounded below filtration of coherent $\O_Y$-modules such that 
\[F_l \M . F_k \sD_Y \subseteq F_{k+l}\M\]
and the equality holds for all $l>l_0$ for some $l_0$. One can associate to such a $\sD_X$-module the de Rham complex
\[\DR_Y(\M) = [\M \otimes_{\O_Y} \bigwedge^{d_Y}T_Y \to \M \otimes_{\O_Y} \bigwedge^{d_Y-1} T_Y \to \dots \to \M]\]
supported in degrees from $-d_Y$ to $0$, in which the differential is induced by the (right) $\sD_Y$-module structure $\M \otimes_{\O_Y} T_Y \to M$. A good filtration on $\M$ defines subcomplexes
\[F_k \DR_Y(\M)= [F_{k-d_Y}\M \otimes_{\O_Y} \bigwedge^{d_Y}T_Y \to F_{k-d_Y+1}\M \otimes_{\O_Y} \bigwedge^{d_Y-1} T_Y \to \dots \to F_k\M]\]
and the associated graded complexes
\[\gr^F_k \DR_Y(\M)= [\gr^F_{k-d_Y}\M \otimes_{\O_Y} \bigwedge^{d_Y}T_Y \to \gr^F_{k-d_Y+1}\M \otimes_{\O_Y} \bigwedge^{d_Y-1} T_Y \to \dots \to \gr^F_k\M].\]
Since some of the results in literature use left $\sD$-modules, we recall the equivalence of categories of left and right $\sD$-modules for the convenience of the readers. To a right $\sD_Y$-module $\M$, one can associate a left $\sD_Y$-module 
\[\M^l = \sHom_{\O_Y}(\omega_Y, \M) \cong \omega_Y^{-1} \otimes_{\O_Y} \M,\] 
and to get back to the right $\sD_Y$-module, we simply tensor with $\omega_Y$. Via this correspondence, the induced good filtration is indexed by the convention
\[F_{k}\M^l = \omega_Y^{-1} \otimes_{\O_Y} F_{k-d_Y}\M.\]
Consequently, we have the following formulas for the associated de Rham complex 
\[\DR_Y(\M^l) = [\M^l \to \M^l \otimes _{\O_Y} \Omega^1_Y \to \dots \to \M^l \otimes_{\O_Y} \omega_Y]\]
and its associated graded pieces
\[\gr^F_k\DR_Y(\M^l) = [\gr^F_{k}\M^l \to \gr^F_{k+1}\M^l \otimes _{\O_Y} \Omega^1_Y \to \dots \to \gr^F_{k+d_Y}\M^l \otimes_{\O_Y} \omega_Y].\]

\medskip

Given an embedding of a variety $X$ into a smooth variety $Y$, a mixed Hodge module $M=(\M, F_\bullet \M, \P, \alpha, W)$ on $X$ consists of a regular holonomic $\sD_Y$-module with a good filtration (Hodge filtration) $F_\bullet \M$, a perverse sheaf $\P$ over $\QQ$ on $X^{\text{an}}$ such that $\P\otimes_{\QQ} \CC \cong \DR_Y(\M)$ via $\alpha$, and an increasing weight filtration $W$. We will abuse the notation by also writing $F_\bullet M$ for the Hodge filtration of the underlying $\D_Y$ module of $M$. These data are supposed to satisfy a complicated set of conditions we will not discuss. We denote by $MHM(X)$ the abelian category of mixed Hodge modules on $X$; this category is independent of the embedding of $X$ into smooth varieties. Furthermore, there is a natural faithful and exact functor 
\begin{equation} \label{rat functor}
    \text{rat}: MHM(X) \to \Perv(\underline \QQ_X)
\end{equation}
where $\Perv(\underline \QQ_X)$ is the category of $\QQ$-perverse sheaves on $X^{an}$ with algebraic stratification and $\text{rat}$ takes the underlying rational structure of a mixed Hodge module. The category $MHM(X)$ contains a subcategory $MH(X,w)^p$ of polarizable Hodge modules of pure weight $w$ \cite{Saito-88}*{5.1.6, 5.2.10}. We say $M$ has strict support $Z$ if the underlying perverse sheaf $\text{rat}(M)$ is an intersection complex with support $Z$. These objects form a full subcategory $MH_Z(X, w)^p$ of $MH(X, w)^p$, and we have the decomposition by strict support
\begin{equation}
    MH(X,w)^p = \oplus_{Z\subseteq X} MH_Z(X,w)^p.
\end{equation}
The Tate twist $M(k)$ of a mixed Hodge module $M$ has the same underlying $\D_Y$-module, but the two filtrations are shifted by 
\[F_p M(k) = F_{p-k} M \quad \text{and} \quad W_q M(k) = W_{q+2k} M\]
for all $k \in \ZZ$. 

For a mixed Hodge module $M$, its de Rham complex is the one for the underlying filtered $\D_Y$-modules. It is well-known that the associated graded complex $\gr^F_p \DR_Y(M)$ is a complex of coherent $\O_X$-modules. Moreover, this is well defined for complexes of mixed Hodge modules $M^\bullet$, thereby inducing the graded de Rham functor 
\[\gr^F_p \DR_Y: D^b MHM(X) \to D_{coh}^b(X)\]
from the bounded derived category of mixed Hodge modules on $X$ to the bounded derived category of coherent $\O_X$-modules. This functor is independent of the embedding of $X$ into smooth varieties (see \cite[Proposition 2.33]{saito-MHM}); and because of that reason, we will omit the subscript and write $\gr^F_p \DR$ for the graded de Rham functor. 

One beautiful accomplishment of Saito's theory of mixed Hodge modules is the construction of natural functors $f_*, f_!, f^*, f^!, \otimes, \sHom, \bD_X$ in $D^bMHM(X)$ such that they are compatible with the corresponding functors on the underlying $\QQ$-complexes via $rat: D^bMHM(X) \to D^b \Perv(\underline \QQ_X)$. Moreover, we emphasize the following facts that will be used throughout the paper (see \cite[Section 4]{saito-MHM}). 

\begin{proposition} \label{proposition: base change and commute with dual}
    Given a morphism $f: X\to Y$ between algebraic varieties, there are canonical isomorphisms:
    \[\bD_X \circ \bD_X \cong Id, \quad \bD_Y \circ f_* \cong f_! \circ \bD_X, \quad \bD_X \circ f^* \cong f^! \circ \bD_Y.\]
    For a cartesian square 
    \[\begin{tikzcd}
	{X'} & X \\
	{Y'} & Y
	\arrow["{g'}", from=1-1, to=1-2]
	\arrow["f", from=1-2, to=2-2]
	\arrow["g"', from=2-1, to=2-2]
	\arrow["{f'}"', from=1-1, to=2-1]
    \end{tikzcd}\]
    we have canonical isomorphisms:
    \[g^! \circ f_* \cong f'_* \circ (g')^!, \quad g^* \circ f_! \cong f'_! \circ (g')^*.\]
    Similarly, these are compatible with the natural isomorphisms on the underlying $\QQ$-complexes. 
\end{proposition}

\begin{remark}
    Note that we have two different duality functors up to now: $\DD_X$ for the bounded derived category of coherent sheaves and $\bD_X$ for the bounded derived category of mixed Hodge modules. 
\end{remark}

\begin{proposition}
    Let $i: Z\to X$ be a closed immersion and $j:U = X \smallsetminus Z \to X$ the immersion of the complement. Then we have the following functorial triangles for any $M^\bullet \in D^bMHM(X)$:
    \[j_! j^* M^\bullet \to M^\bullet \to i_* i^* M^\bullet \xto{+1}\]
    and 
    \[i_* i^! M^\bullet \to M^\bullet \to j_* j^* M^\bullet \xto{+1}\]
    compatible with those for the underlying $\QQ$-complexes. 
\end{proposition}

The proof of the main results relies on an in-depth understanding of the graded de Rham functor $\gr^F_p \DR$. Here, we recall the behaviors of this functor with respect to the natural functors in $ D^bMHM (X)$ in the most general form (see \cite{Saito-88}, \cite{Saito-MHC}, \cite{Sunggi-DB-extension-forms})

\begin{proposition} \label{proposition: grDR and f_* commute}
    Let $f: X\to Y$ be a proper morphism of algebraic varieties. For any $M^\bullet \in D^b MHM(X)$, we have natural isomorphisms 
    \[\gr^F_p \DR ( f_* (M^\bullet)) \cong Rf_* ( \gr^F_p \DR(M^\bullet))\]
    in $D_{coh}^b(Y)$, and 
    \[\DD_X(\gr^F_p DR(M^\bullet)) \cong \gr^F_{-p} DR (\bD_X(M^\bullet))\]
    in $D_{coh}^b(X)$. 
\end{proposition}

\subsection{Decomposition theorem and Torsion-freeness theorem for Hodge modules of pure weight}

Given a projective variety $X$, for any $M^\bullet\in D^b MHM(X)$, we define 
\[p(M^\bullet) \coloneqq \min \{p\  |\  \gr^F_p \DR(M^\bullet) \text{ is not acyclic}\}.\] 
If we fix an embedding of $X$ into a smooth variety $Y$ and $M$ is a complex of mixed Hodge module on $Y$ supported on $X$, then $p(M)$ is exactly the index of the lowest nonzero Hodge filtration.

Now, we recall Saito's results, generalizing \Kollar{}'s decomposition and torsion freeness theorem.

\begin{theorem} [Saito's decomposition theorem] \label{theorem: saito decomposition} Let $f: X\to Y$ be a projective morphism between two smooth complex varieties. Let $M\in MH(X, w)^p$ be a polarizable Hodge module of pure weight $w$. Then each $H^i f_* M \in MH(Y, w+i)^p$ is of pure weight $w+i$ and admits a non canonical isomorphism 
\[f_* M \cong \bigoplus_{i} H^i(f_* M)[-i].\]
\end{theorem}

Because of the strict support decomposition
\[MH(X, w)^p = \bigoplus_Z MH_Z(X, w)^p,\]
each $H^i(f_* M)$ admits further decomposition by strict support. Due to the commutativity of $\gr^F_p \DR$ and $f_*$, we get $p(H^i f_* M) \ge p(f_*M) \ge p(M)$. It turns out that we can detect the inequality of lowest Hodge indexes of the decomposition by the strict support. 

\begin{theorem} [{\cite[Proposition 2.6]{Saito_Kollar_conjecture}, \cite{Saito_Kollar_conjecture}*{Theorem 3.2}}] \label{theorem: saito torsion-freeness}
    Let $f: X\to Y$ be a projective morphism between two complex algebraic varieties and $M\in MH_{Z'}(X, w)^p$ where $Z'\subseteq X$. Let $H^i f_* M = \bigoplus_{Z} M^i_{Z}$ be the decomposition by strict support. Then $p(M^i_Z) > p(M)$ if $Z \neq f(Z')$.

    As a consequence, if $M\in MH_X(X, w)^p$, then 
    \[Rf_* \gr^F_{p(M)} \DR (M) \cong \bigoplus_{p(M^i_{f(X)}) = p(M)} \gr^F_{p(M)} \DR (M^i_{f(X)})[-i].\]
\end{theorem}

When $f:X \to Y$ is a surjective map with $X$ smooth and $M = \QQ^H_X[d_X]$ the trivial mixed Hodge module (see Section \ref{section: Du Bois and intersection complex} for the definition), Theorem \ref{theorem: saito torsion-freeness} says that 
\[Rf_* \omega_X \cong \bigoplus_{i} \gr^F_{-d_X} \DR (M^i_{Y})[-i]\]
and that $p(H^i f_* \QQ_X^H[d_X]) \ge -d_X$ and
\[\gr^F_{-d_X} \DR(M^i_Y) = gr^F_{-d_X} \DR(H^if_* \QQ^H_X[d_X]) = R^if_* \omega_X\]
by Proposition \ref{proposition: grDR and f_* commute} and the strictness of the Hodge filtration. This recovers \Kollar{}'s decomposition and torsion-freeness theorem which states that 
\begin{enumerate}
    \item $R^if_* \omega_X$ is torsion-free.
    \item $Rf_* \omega_X \cong \bigoplus_i R^if_*\omega_X[-i]$.
\end{enumerate}

\subsection{Intersection complex and Du Bois complex} \label{section: Du Bois and intersection complex}

Mixed Hodge modules are defined inductively on the dimension of the variety $X$. In the base case, when $X = \Spec (\CC)$ is a point, the category $MHM(pt)$ is equivalent to the category of $\QQ$-mixed Hodge structures (see \cite[Theorem 3.9]{saito-MHM}). In particular, there is a trivial mixed Hodge module on a point 
\[\QQ^H_{pt} \coloneqq (\CC, F_\bullet \CC, \QQ; W)\]
where $F_0 \CC = \CC,\  F_{-1} \CC = 0$ and $W_0 \QQ = \QQ,\  W_{-1} \QQ = 0$; it corresponds to the trivial Hodge structure $\QQ$ of weight $0$. 

Let $X$ be an algebraic variety. The trivial mixed Hodge module on $X$ is defined as
\[\QQ^H_X \coloneqq a_X^* \QQ^H_{pt}\]
where $a_X: X \to pt$. When $X$ is smooth, one can easily see that 
\[\QQ^H_X[d_X] = (\omega_X, F_\bullet \omega_X, \underline{\QQ}_X[d_X])\]
is a polarizable Hodge module of pure weight $d_X$ with $F_{-d_X} \omega_X = \omega_X$ and $F_{-d_X-1} \omega_X = 0$. In general, the underlying $\QQ$-complex of $\QQ^H_X$ is the constant sheaf $\underline \QQ_X$, and the cohomologies of the direct image $a_{X*} \QQ^H_X \in D^bMHM(pt)$ give us the mixed Hodge structures on the singular cohomologies of $X$ with rational coefficient. Since the functor $\text{rat}$ is faithful and $\prescript{p}{}{\H}^i(\underline \QQ_X[d_X]) = 0$ for $i>0$, we get
\[H^i \QQ_X^H = 0 \quad \text{for} \quad i > d_X.\]
We say that $M^\bullet \in D^b MHM(X)$ is of weight $\le w$ (resp. $\ge w$) if 
\[\gr^W_q H^iM = 0 \quad \text{for} \quad q > w +i \text{ (resp. } q<w+i).\]
Moreover, this condition is stable under the functors $f_!, f^*$ (resp. $f_*, f^!$). As a consequence, the trivial mixed Hodge module $\QQ^H_X$ is of weight $\le 0$; in particular, 
\[\gr^W_i H^{d_X} \QQ^H_X = 0 \quad \text{for} \quad i>d_X.\]
If $X$ is irreducible or equidimensional, the boundary term $\gr^W_{d_X} H^{d_X} \QQ^H_X$ is exactly the \textit{intersection complex} $\IC_X \QQ^H$, which, by definition, is the unique object in $MHM(X)$ such that its restriction to $X_{reg}$ is $\QQ^H_{X_{reg}}[d_X]$ and it has no subobject or quotient in $MHM(X)$ whose support is contained in $X_{\sing}$. We know that $\IC_X \QQ^H$ is simple and pure of weight $d_X$, and its underlying $\QQ$-complex is the well-known intersection cohomology complex $\IC_X^\bullet$ defined as the intermediate extension of constant sheaf $\underline \QQ_{X_{\reg}}$ on the smooth locus of $X$. Similar to the self-duality of $\IC_X^\bullet$, the intersection complex is self-dual up to a Tate twist
\[\bD_X (\IC_X\QQ^H) \cong \IC_X \QQ^H (d_X).\]
We refer the reader to \cite[Section 4.5]{saito-MHM} for details. 

\medskip

Now, we focus on the object of our interest - the (intersection) Du Bois complex. Here we only recall important facts and refer the readers to \cite{DB-injectivity} and \cite{Mihnea-Sunggi} for detailed discussion of (intersection) Du Bois complexes. For a smooth projective variety $X$, the de Rham complex is a resolution of constant sheaf $\underline \CC_X$ by sheaves of differential forms
\[\underline \CC_X \to \O_X \xto{d} \Omega_X^1 \xto{d} \dots \xto{d} \Omega^{d_X}_X.\]
The filtration by truncation $F^p\Omega^\bullet_X = \Omega^{\bullet \ge p}_X$
induces the Hodge structures on the singular cohomology of $X$. In \cite{DB}, Du Bois constructed the Du Bois complex $\DB^\bullet_X$ via hyperresolution as an analog of the de Rham complex for singular varieties. It is an object in the derived category of filtered complexes on $X$, quasi-isomorphic to the constant sheaf $\underline \CC_X$. The $p$-th (shifted) associated graded piece 
\[\DB^p_X \coloneqq \gr^p_F \DB^\bullet_X[p]\]
is an object in $D_{coh}^b(X)$, which we call \textit{$p$-th Du Bois complex}. For any $p$, there is a natural map 
\[\Omega^p_X \to \DB^p_X\]
which is an isomorphism if $X$ is smooth. Even though Du Bois complexes are complicated objects, the top Du Bois complex $\DB_X^{d_X}$ can be described simply due to \cite{Steenbrink-vanishing}*{Proposition 3.3} as:
\[\DB_X^{d_X} \cong Rf_* \omega_{\tilde X} \cong f_* \omega_{\tilde X}\]
where $\tilde X$ is a log resolution of $X$ and the second isomorphism follows from Grauert-Riemenschneider vanishing theorem. From the perspective of mixed Hodge modules, as a consequence of \cite[Theorem 4.2, Corollary 0.3]{Saito-MHC}, Saito shows that 
\begin{equation} \label{DB as grDR}
    \DB^p_X \cong \gr^F_{-p} \DR(\QQ^H_X[d_X])[p-d_X]
\end{equation}
for any $p$, giving mixed Hodge modules interpretation for Du Bois complexes. Similarly, we define
\begin{equation} \label{IC as grDR}
    I\DB_X^p \coloneqq \gr^F_{-p} \DR(\IC_X\QQ^H)[p-d_X]
\end{equation}
for intersection complex and call this \textit{intersection Du Bois complex}. The intersection Du Bois complexes generalize \Kahler{} differentials while maintaining the \Poincare{} duality. More precisely, the self-duality $\bD_X(\IC_X\QQ^H)\cong \IC_X\QQ^H(d_X)$ implies that 
\[\DD_X(I\DB_X^k) \cong I\DB_X^{d_X-k}[d_X].\]
The natural morphism $\QQ^H_X[d_X] \to \gr^W_{d_X} H^{d_X} \QQ^H_X \cong \IC_X\QQ^H$ induces canonical morphisms
\[\DB^p_X \to I\DB_X^p\]
and dually
\[I\DB_X^p \to \DD_X(\DB^{d_X-p}_X)[-d_X].\]
Moreover, the definitions of (intersection) Du Bois complexes imply that $p(\QQ_X^H[d_X]) = p(\IC_X\QQ^H) = -d_X$ and $p(\bD(\QQ_X^H[d_X])) = 0$.

\subsection{Local cohomological dimension}

An important invariant in the study of Du Bois complexes is the local cohomological dimension. If a variety $X$ is embedded into a smooth variety $Y$, the local cohomological dimension of $X$ in $Y$, considered by Ogus in \cite{Ogus-lcd}, is defined as 
\begin{equation} \label{definition: lcdef}
    \lcd(X,Y) \coloneqq \max\{i|\  \H^i_X(F) \neq 0\ \text{ for any coherent sheaf $F$}\}.
\end{equation}
By \cite[Proposition 2.1]{Ogus-lcd}, it suffices to consider $F=\O_Y$ in the definition of $\lcd(X,Y)$. Moreover, it is known that $\H^i_X(\O_Y)$ vanishes when $i$ lies outside the range $\codim(X,Y)\le i \le \lcd(X,Y)$ (see \cite{Ogus-lcd}*{Theorem 2.8}). We recall several characterizations of $\lcd(X,Y)$ in terms of Du Bois complexes of $X$ and the trivial Hodge module $\QQ^H_X[d_X]$. 

\begin{proposition}[{\cite[Corollary 12.6]{MP-lci}}] \label{proposition: lcd and DB}
    For every positive integer $c$, the following are equivalent:
    \begin{enumerate}
        \item $\lcd(X,Y) -d_Y \le c$.
        \item $\sExt^{j+p+1}_{\O_X}(\DB^p_X, \omega_Y^{\bullet})=0$ for all $j\ge c$ and $p\ge 0$.
    \end{enumerate}
\end{proposition}

By using Proposition \ref{proposition: grDR and f_* commute}, it is not hard to obtain the following characterization of $\lcd(X,Y)$:

\begin{theorem}[{\cite{Ogus-lcd}*{Theorem 2.13}, \cite[Theorem 1]{saito-lcd}, \cite{Bhatt}*{Section 3.1}}] \label{theorem: cohom range of Q_X} With the above notations, 
\[\codim(X,Y)-\lcd(X,Y) = \min\{i|H^i(\QQ^H_X[d_X])\neq 0\}.\] 
\end{theorem}

The \textit{local cohomological defect} of $X$ is defined to be
\[\lcdef(X) \coloneqq \lcd(X,Y) - \codim(X,Y).\]
This quantity is independent of the embedding $X\subset Y$ by Proposition \ref{proposition: lcd and DB}. Moreover,
\[0\le \lcdef(X)\le d_X\] 
since it is well-known that $\H^i_X(\O_Y)$ = 0 for $i > d_Y$. 

\begin{remark}
    $\lcdef(X) = 0$ is equivalent to say that $\codim(X,Y) = \lcd(X,Y)$; in other words, the only nonzero local cohomology is $\H^{\codim(X,Y)}_X(\O_Y)$. Obviously, smooth varieties have zero local cohomological defect. This also holds for local complete intersections or rational homology manifolds.      
\end{remark}

\subsection{Integral functors and GV-objects} \label{section: integral functors} 
First, we recall the definition of cohomological support loci. Let $X$ and $Y$ be projective varieties, and $P$ a perfect object on $D_{coh}^b(X\times Y)$ (i.e., an object represented by a bounded complex of locally free sheaves of finite rank). For any $F \in D_{coh}^b(X)$, we denote by 
\[V^i_P(F) \coloneqq \{y\in Y\ |\ H^i(X, F\otimes^\LL Li_{y}^* P) \neq 0\}\]
the \textit{$i$-th cohomological support locus} of $F$ with respect to $P$, where $i_y: X \times \{y\} \hookrightarrow X\times Y$. The content of the generic vanishing theory is to study the cohomological support locus, when $Y = \Pic^0(X)$ and $P=P_X$ is the universal line bundle; thanks to many people, the structure of these sets is well-understood in the following aspects: 
\begin{itemize}
    \item One has $\codim V^i_{P_X}(\omega_X)\ge i-\dim X + \dim a(X)$ for all $i$ (see \cite{GL87}, \cite{GL91}). This implies the classical generic vanishing of Green and Lazarsfeld by applying Serre duality.
    \item The irreducible components of each $V^i_{P_X}(\omega_X)$ are torsion translates of abelian subvarieties of $\Pic^0(X)$ (see \cite{GL87}, \cite{GL91}, \cite{Sim-rank1_locsys})
\end{itemize}

In \cite{Hacon-gen_vanishing}, Hacon observes that generic vanishing property is closely related to the Fourier-Mukai transforms studied by Mukai in \cite{Mukai_transform}. Let us recall this relation in the general setting. Let $p_X$ and $p_Y$ are the projections from $X\times Y$ to $X$ and $Y$ correspondingly. We define the \textit{integral functors}:
\[R\Phi_P:D_{coh}^b(X) \to D_{coh}^b(Y), \quad R\Phi_P(\cdot) \coloneqq Rp_{Y*}(Lp_X^* (\cdot) \otimes^\LL_{\O_{X\times Y}} P)\]
and 
\[R\Psi_P:D_{coh}^b(Y) \to D_{coh}^b(X), \quad R\Psi_P(\cdot) \coloneqq Rp_{X*}(Lp_Y^*(\cdot) \otimes^{\LL}_{\O_{X\times Y}} P)\]
Although the locus $\supp(R^i\Phi_{P}(F))$ is only contained inside $V^i_P(\F)$, by a standard cohomological base change argument, the two loci carry the same numerical information in the sense that for any nonnegative integer $k$, the following are equivalent \cite[Lemma 3.6]{PP_Generic_Vanishing}:
\begin{enumerate}
    \item $\codim(V^i_P(F), Y) \ge i-k$ for all $i \ge 0$.
    \item $\codim(\supp R^i\Phi_{P}(F), Y) \ge i-k$ for all $i\ge 0$. 
\end{enumerate}
Thus one can study the support of integral functors to gain information of the dimension of cohomological support loci. It is then natural to introduce the following notions (see \cite{PP_Generic_Vanishing}*{Section 2 and 3} for details):

\begin{definition}                                   
    For any integer $k\ge 0$, an object $F\in D_{coh}^b(X)$ is said to be a \textit{$GV_{-k}$-object with respect to $P$} if                             
    \[\codim(\supp R^i\Phi_{P}(F), Y) \ge i-k \quad \text{for all } i\ge 0\]        
    When $k=0$, we simply say $F$ is a $GV$-object.                            
\end{definition}                       

Another key concept in Fourier-Mukai theory is that of an object satisfying $WIT$ (Weak Index Theorem) on abelian varieties. This property is generalized to the theory of integral functors:

\begin{definition}                  
    An object $F \in D_{coh}^b(X)$ satisfies \textit{Weak Index Theorem with index $j$}, with respect to $P$, if $R^i\Phi_P(F) = 0$ for all $i\neq j$. For simplicity, we also say that $F$ is $WIT_{j}$ (with respect to $P$).             
    
    More generally, we say $F$ is $WIT_{\ge j}$ (resp. $WIT_{\le j}$) if $R^i\Phi_P(F) = 0$ for all $i<j$ (resp. $i>j$) (with respect to $P$).                
\end{definition}                   

\begin{remark} \label{GV and WIT with shift}
    From the definitions, it is easy to check that                   
\begin{align*}                    
    \begin{split}             
        &F \text{ is } GV_{-k} \text{ if and only if } F[j] \text{ is } GV_{j-k} \text{ and }\\                              
        &F \text{ is } WIT_{\ge k} \text{ if and only if } F[j] \text{ is } WIT_{\ge k-j} \text{ for any } j\le k.                         
    \end{split}           
\end{align*} 
\end{remark}

Let $f:X'\to X$ be a morphism from a projective variety $X'$ to $X$ and $P'\coloneqq L(f\times id_Y)^*(P)$ a perfect object in $D_{coh}^b(X'\times Y)$. Assume further that $Y$ is smooth; this condition is satisfied, for instance, when $X$ is normal or abelian variety and $Y= \Pic^0(X)$. Consider the following commutative diagram:
\[\begin{tikzcd}                           
	{X'} && {X'\times Y} \\
	\\
	X && {X\times Y} && Y
	\arrow["f"', from=1-1, to=3-1]
	\arrow["{p_{X'}}"', from=1-3, to=1-1]
	\arrow["{p_X}"', from=3-3, to=3-1]
	\arrow["{f\times id_Y}"{description}, from=1-3, to=3-3]
	\arrow["{p_Y}"', from=3-3, to=3-5]
	\arrow["{p_Y'}", from=1-3, to=3-5]
\end{tikzcd}\]
For any $F\in D_{coh}^b(X')$, by projection formula and base change, we get 
\begin{align*}
    R\Phi_P(Rf_* F) &= Rp_{Y*}(Lp_X^* (Rf_*F) \otimes^\LL P) \\
    &\cong Rp_{Y*}\big (R(f\times id_Y)_* (Lp_{X'}^*F)\otimes^{\LL} P\big) \quad (\text{by base change}) \\
    &\cong Rp'_{Y*}(Lp_{X'}^* F \otimes^{\LL} P') \quad (\text{by projection formula}) \\
    &= R\Phi_{P'}(F)
\end{align*}
Thus we can say $F$ is $GV_{-k}$ (resp. $WIT_{\ge k}$) on $X'$ (with respect to $P'$) and $Rf_*F$ is $GV_{-k}$ (resp. $WIT_{\ge k}$) on $X$ (with respect to $P$) interchangeably.

\medskip

The $GV$ and $WIT$ properties are closely related via duality. Note that by base change, for any sufficiently ample line bundle $L$ on $Y$, we have $L^{-1}$ is $WIT_{d_Y}$ on $X$ with respect to $P$ and $R^{d_Y}\Psi_P(L^{-1})$ is locally free sheaf on $X$. Using this, $GV$ and $WIT$ properties can be checked by the vanishing of some hypercohomology group as follows.

\begin{theorem}[{\cite{Hacon-gen_vanishing}*{Theorem 1.2}},{\cite[Theorem 3.7]{PP_Generic_Vanishing}}] \label{theorem: GV and WIT equiv}
    Let $X$ and $Y$ be projective Cohen-Macaulay varieties and $P$ a perfect object in $D_{coh}^b(X\times Y)$. For $F\in D_{coh}^b(X)$ and $k\ge 0$, the following are equivalent:
    \begin{enumerate}
        \item $F$ is a $GV_{-k}$-object with respect to $P$. 
        \item $\DD_X(F)$ is $WIT_{\ge -k}$ with respect to $P^\vee \otimes^{\LL} Lp_Y^* \omega_Y$ where $P^\vee = R\sHom(P,\O_{X\times Y})$. 
        \item $\HH^i(X, F\otimes^{\LL}_{\O_X} R^{d_Y}\Psi_P(L^{-1}))=0$ for all $i>k$ and for any sufficiently ample line bundle $L$ on $Y$. 
    \end{enumerate} 
\end{theorem}

With these notations, the classical generic vanishing states that $\omega_X$ is $GV_{d_{a(X)}-d_X}$, or equivalently, $\O_X$ is $WIT_{\ge d_{a(X)}}$. 

\begin{corollary} \label{corollary: GV exact sequence}
    Let $F\to G\to H\xto{+1}$ be a distinguished triangle in $D_{coh}^b(X)$. Suppose $F$ is $GV_{-k}$ with respect to $P$ for some $k\ge 0$. Then $G$ is $GV_{-k}$ if and only if $H$ is $GV_{-k}$ with respect to the same perfect object $P$. 
\end{corollary}

\begin{proof}
    Since $R^{d_Y}\Psi_P(L^{-1})$ is locally free, tensoring by $R^{d_Y}\Psi_P(L^{-1})$ is exact; so we obtain a long exact sequence 
    \begin{multline*}
    \dots \to \HH^i(X, F \otimes^{\LL} R^{d_Y}\Psi_P(L^{-1})) \to \HH^i(X, G \otimes^{\LL} R^{d_Y}\Psi_P(L^{-1})) \to \\
    \to\HH^i(X, H \otimes^{\LL} R^{d_Y}\Psi_P(L^{-1})) \to 
    \HH^{i+1}(X, F \otimes^{\LL} R^{d_Y}\Psi_P(L^{-1}))\to \dots.
    \end{multline*}
    The statement now follows from Theorem \ref{theorem: GV and WIT equiv}. 
\end{proof}

\subsection{Generic vanishing of mixed Hodge modules on abelian varieties}

On an abelian variety $A$, any ample line bundle $L$ induces an isogeny $\phi_L: \hat A \to A$. This property was used in \cite{PS_generic_vanishing} to obtain $GV$ properties of mixed Hodge modules on $A$. Since $P_A^\vee = (1\times(-1))^*P_A$ for any abelian variety $A$, we will say $GV_{-k}$ or $WIT_{\ge -k}$ on $A$ without specifying the perfect object. 

\begin{theorem}[{\cite[Corollary 9.3]{PS_generic_vanishing}}] \label{theorem: GV for MHM}
    Let $A$ be an abelian variety, and $M$ a mixed Hodge module on $A$. Then we have $\gr^F_k(M)$ and $\gr^F_k \DR(M)$ are $GV$ on $A$ for any $k$. 
\end{theorem}

Note that $\bD_A(M)$ is still a mixed Hodge module on $A$; thus we can apply Theorem \ref{theorem: GV for MHM} to $\bD_A(M)$ to obtain $\DD_A(\gr^F_{k} \DR(M)) \cong \gr^F_{-k} \DR(\bD_A(M))$ is $GV$ on $A$. By Theorem \ref{theorem: GV and WIT equiv}, this means $\gr^F_k \DR(M)$ is both $GV$ and $WIT_{\ge 0}$ on $A$ for any $k$. 
With a simple argument, we obtain $GV$ properties for objects in $D^bMHM(A)$ based on their cohomological degrees.

\begin{corollary} \label{corollary: GV for D^bMHM}
    If $M^\bullet \in D^bMHM(A)$ with $H^i(M^\bullet) = 0$ for $i<-a$ or $i>b$ for non-negative integers $a$ and $b$, then $\gr^F_{k} \DR(M^\bullet)$ is $GV_{-b}$ and $\DD_A(\gr^F_k \DR(M^\bullet))$ is $GV_{-a}$ on $A$.  
\end{corollary}

\begin{proof}
    Since $\DD_A(\gr^F_k \DR(M^\bullet)) \cong \gr^F_{-k} \DR(\bD_A(M^\bullet))$ and $\H^i(\bD_A(M^\bullet)) = 0$ for $i<-b$ or $i>a$, it suffices to show that $\gr^F_k \DR(M^\bullet)$ is $GV_{-b}$. By considering $M^\bullet[-a]$, we can assume $a=0$. With this assumption, we prove the first statement by induction on $b$. The base case, $b=0$, is exactly Theorem \ref{theorem: GV for MHM}. If $b>0$, consider the distinguished triangle
    \[\gr^F_{k} \DR(\tau^{<b}M^\bullet) \to \gr^F_{k} \DR(M^\bullet) \to \gr^F_{k} \DR(H^b(M^\bullet))[-b] \xto{+1}.\]
    By induction hypothesis, we have $\gr^F_{k} \DR(\tau^{<b} M^\bullet)$ is $GV_{-b+1}$ and hence is $GV_{-b}$ on $A$. Moreover, the last term $\gr^F_{k} \DR(H^b(M^\bullet))[-b]$ is also $GV_{-b}$; thus Corollary \ref{corollary: GV exact sequence} implies that $\gr^F_{k} \DR(M^\bullet)$ is $GV_{-b}$ on $A$, as we want.
\end{proof}

Now, let's briefly recall Hacon's proof of the classical generic vanishing theorem. Since 
\[Ra_* \omega_X \cong \oplus_i R^i a_* \omega_X[-i]\]
By Corollary \ref{corollary: GV for D^bMHM}, it is enough to show that $R^i a_* \omega_X = 0$ if $i > d_{X} - d_{a(X)}$, which follows from \Kollar{}'s torsion-freeness theorem that $R^i a_* \omega_X$ is torsion-free. Analogously, we have

\begin{corollary} \label{corollary: GV of lowest Hodge filtration of pure HM}
    Let $f: X\to A$ be a projective morphism to an abelian variety and $M \in MH_X(X, w)^p$ a polarizable Hodge module of pure weight $w$ with strict support $X$. Then $Rf_* \gr^F_{p(M)} \DR(M)$ is $GV_{d_{f(X)}-d_X}$. 
\end{corollary}

\begin{proof}
    By Theorem \ref{theorem: saito torsion-freeness}, we have 
    \[Rf_* \gr^F_{p(M)} \DR (M) \cong \bigoplus_{p(M^i_{f(X)}) = p(M)} \gr^F_{p(M)} \DR (M^i_{f(X)})[-i].\]
    Note that if $H^if_* M$ contains $M^i_{f(X)}$ as a direct summand, then $i\le d_X - d_{f(X)}$. Thus $\gr^F_{p(M)} \DR (M^i_{f(X)})[-i]$ is $GV_{d_{f(X)}-d_X}$ by Corollary \ref{corollary: GV for D^bMHM}, which implies $Rf_* \gr^F_{p(M)} \DR(M)$ is $GV_{d_{f(X)}-d_X}$.
\end{proof}

\medskip

\section{Generalization of classical generic vanishing for singular varieties}

Before discussing the main theorems, Corollary \ref{corollary: GV for D^bMHM} suggests that we need to study the cohomological amplitude of direct image of $\QQ_X[d_X]$ and $\IC_X\QQ^H$. 

\subsection{(Stratified) semismallness defect}

The semismallness defect is an important invariant of singularities that controls the generic vanishing of Nakano type \cite{PS_generic_vanishing}. This explains the counter-example of Green and Lazarsfeld that the index in the generic Nakano vanishing theorem is not equal to the dimension of the generic fiber of the Albanese mapping (see \cite{PS_generic_vanishing}*{Example 12.3}). When $X$ is smooth, this invariant is studied by de Cataldo and Migliorini (see \cite[Remark 2.1.2]{CM-semismalldefect} and \cite[Proposition 11.2]{PS_generic_vanishing})

\begin{definition} [{\cite[Definition 4.7.2]{CM-semismalldefect}}] \label{definition: semismall-defect}
    The \textit{semismallness defect} of a morphism $a:X\to A$ is 
    \[\delta(a) \coloneqq \max_{l\in \NN} \{2l-\dim(X) + \dim(A_l)\}\]
    where $A_l = \{y\in A | \dim a^{-1}(y)\ge l \}$.
\end{definition}

Their result says that 
\[\delta(a) = \max \{H^i a_* \QQ^H_X[d_X] \neq 0\} = -\min \{H^i a_* \QQ^H_X[d_X] \neq 0\}\] in which the second equality is due to the self-duality of $\QQ^H_X[d_X] \cong \IC_X\QQ^H$, given the smoothness of $X$. By using this characterization, Popa and Schnell show in \cite[Theorem 3.2]{PS_generic_vanishing} that 
\[\Omega^p_X \text{ is } GV_{p-d_X-\delta(a)} \text{ for every } p\in \NN\]
or equivalently, 
\[\codim V^q(\Omega^p_X) \ge p+q-d_X -\delta(a)\] 
for every $p,q\in \NN$, and there exist $p$ and $q$ for which the equality holds.  

When $X$ is singular, the semismallness defect no longer controls the cohomological amplitude of $a_* \QQ^H_X[d_X]$. For example, consider an embedding of $X$ into an abelian variety $A$, in which case $\delta(a) = 0$ and $\lcdef(X) > 0$, but we have seen in Theorem \ref{theorem: cohom range of Q_X} that $H^{-\lcdef(X)} \QQ^H_X[d_X] \neq 0$. The issue is that $\delta(a)$ does not take into account the interaction of the singularity of $X$ and the morphism $a:X \to A$. To remedy this, we consider Whitney stratifications of $X$, which stratify $X$ into smooth \quotes{equisingular loci} (see \cite[Definition E.3.7]{HTT} for precise definition). In \cite{min-whitney-strata}, Teissier proves the existence of a minimal Whitney stratification $X = \bigsqcup_{s\in \sS} X_s$, where $\sS$ is the index set. In other words, any Whitney stratification of $X$ is a refinement of this minimal one. This gives us a canonical choice of Whitney stratification. 

\begin{definition} \label{definition: stratified semismall defect}
    The \textit{stratified semismallness defect} of a morphism $a: X\to A$ is 
    \[\delta_s(a) \coloneqq \max_{s\in \sS} \{\delta(a|_{X_s})\}\]
    where $X = \bigsqcup_{s\in \sS} X_s$ is the minimal Whitney stratification of $X$. 
\end{definition}

Clearly, if $X$ is smooth, we have $\delta_s(a) = \delta(a)$ since $X$ itself is the minimal Whitney stratification. The following Proposition allows us to control the cohomological amplitude of the direct image of constructible sheaves with respect to the minimal  Whitney stratification. 

\begin{proposition} \label{proposition: stratified semismall defect bound}
    Let $a: X\to A$ be a morphism and $X = \bigsqcup_{s\in \sS} X_s$ the minimal Whitney stratification of $X$. Then for any $F\in D^b_{\sS}(X)$ i.e., $F|_{X_s}$ has locally constant cohomology sheaves (of finite type) for all $s\in \sS$ (see \cite[Definition 2.3.7]{Achar-pervers-sheaves}), we have 
    \begin{enumerate}
        \item If $F\in \Dp^b_{\sS}(X)^{\le 0}$, then $a_! F \in \Dp^b_{\sS}(X)^{\le \delta_s(a)}$.
        \item If $F\in \Dp^b_{\sS}(X)^{\ge 0}$, then $a_* F \in \Dp^b_{\sS}(X)^{\ge -\delta_s(a)}$.
    \end{enumerate}
\end{proposition}

Before proving this proposition, let us consider the following lemma:

\begin{lemma} \label{lemma: coh amp of image of loc sys}
    Let $a:X \to A$ be a morphism from $X$ to $A$, where $X$ is smooth. For any local system $L$ on $X$, we have
    \begin{enumerate}
        \item $\prescript{p}{}{\H}^k a_! L[d_X] = 0$ for $k > \delta(a)$.
        \item $\prescript{p}{}{\H}^k a_* L[d_X] = 0$ for $k < -\delta(a)$.
    \end{enumerate}
\end{lemma}

\begin{proof}
    Consider a stratification $A = \bigsqcup_{t\in \sT} A_t$ with embeddings $i_{A_t}: A_t \to A$ such that $i^{-1}_{A_t} a_!L$ have locally constant cohomology sheaves and all fibers of $a$ over $A_t$ have the same dimension. By \cite[Proposition 8.1.22(i)]{HTT}, the first statement is equivalent to
    \[\H^j (i^{-1}_{A_t} a_! L[d_X]) = 0\]
    for any $t\in \sT$ and $j > -d_{A_t} + \delta(a)$. Note that each stratum $A_t$ is contained in some $A_d \coloneqq \{p\in A| \dim a^{-1}(p) = d\}$. Since pullback is exact, it is enough to show that 
    \[(\H^j a_! L[d_X])_p = 0\]
    for any $p\in A_t$ and $j > -d_{A_t} +\delta(a)$. By base change theorem (Proposition \ref{proposition: base change and commute with dual}), we get
    \[(\H^j a_! L[d_X])_p \cong H^{j+d_X}_c(a^{-1}(p), L) = 0\]
    for $j> 2\dim a^{-1}(p) -d_X=2d-d_X$. By the definition of the semismall defect, we have $2d - d_X \le  -d_{A_t} + \delta(a)$, which concludes the proof.

    Since $X$ is smooth, the second statement follows from the first one by Verdier duality. 
\end{proof}

\begin{remark}
    \begin{enumerate}
        \item When $X$ is smooth and $a$ is proper, this lemma gives the expected bound in de Cataldo and Migliorini's result for the cohomological amplitude of $a_* \QQ^H_X[d_X]$ due to $a_! = a_*$. 
        
        \item \label{remark: semismallness upperbound for Q^H_X} In the proof of Lemma \ref{lemma: coh amp of image of loc sys}, we only use the assumption $X$ smooth for the second statement. Thus when $X$ is singular and $a$ is proper, we get
        \[\prescript{p}{}{\H}^k a_* L[d_X] = \prescript{p}{}{\H}^k a_! L[d_X] = 0 \quad \text{ for } \quad k>\delta(a).\]
    \end{enumerate}
\end{remark}

\begin{proof}[Proof of Proposition \ref{proposition: stratified semismall defect bound}]
    By inductively using the distinguished triangle
    \[j_!j^* F \to F \to i_*i^* F \xto{+1}\]
    for the inclusions of strata $(X_s)_{s\in \sS}$ into its closure, we observe that $\Dp^b_{\sS}(X)^{\le 0}$ is generated under extension by objects of the form $j_{s!} L[n]$ where $L$ is a local system (of finite type) on $X_s$, $n\ge \dim {X_s}$ and $j_s: X_s \to X$ is the embedding of stratum $X_s$. Note that $j_{s!} L[n] \in \Dp^b_{\sS}(X)^{\le 0}$ because $\sS$ is Whitney stratification and hence a good stratification (see \cite[Lemma 2.3.22]{Achar-pervers-sheaves}). So it suffices to prove the above claim for $F = j_{s!} L[n]$. Since each $X_s$ is smooth, we have
    \[a_! j_{s!} L[n] = (a|_{X_s})_! L[n] \in \Dp^b_c(Y)^{\le \delta(f|_{X_s})}\]
    by Lemma \ref{lemma: coh amp of image of loc sys}. Because $\delta_s(a) \ge \delta(f|_{X_s})$, the first statement is proven. 

    For the second statement, by using the distinguished triangle
    \[i_*i^! F \to F \to j_*j^* F \xto{+1}\]
    we know that $\Dp^b_{\sS}(X)^{\ge 0}$ is generated under extension by objects of the form $j_{s*} L[m]$ where $m \le \dim(X_s)$. The remaining proof is analogous. 
\end{proof}

\begin{corollary} \label{corollary: cohomological amplitude bound}
    Let $a:X \to A$ be a morphism between projective varieties. Then 
    \begin{enumerate}
        \item $H^i a_* \QQ^H_X[d_X] = 0$ if $i> \delta(a)$ or $i< -\delta_s(a) -\lcdef(X)$. 
        \item $H^i a_* \IC_X\QQ^H = 0$ if $|i| > \delta_s(a)$. 
    \end{enumerate}
\end{corollary}

\begin{proof}
    Because both $\underline \CC_X[d_X]$ and $\IC_X^\bullet$ belong to $D^b_{\sS}(X)$, the statement follows from Proposition \ref{proposition: stratified semismall defect bound}, Theorem \ref{theorem: cohom range of Q_X} and \ref{rat functor}. 
\end{proof}

\medskip

\subsection{Proof of Main Theorems} 

\begin{proposition}[Theorem \ref{theorem: easy GV for DB^p and IC^p}]
    Let $a:X\to A$ be a morphism from a projective variety $X$ to an abelian variety $A$. Then
    \begin{enumerate}
        \item $Ra_* \DB^{d_X}_X$ and $Ra_* I\DB_X^{d_X}$ are $GV_{d_{a(X)}-d_X}$. 
        \item $Ra_* \DB^{p}_X$ is $GV_{p-d_X -\delta(a)}$ and $Ra_* I\DB_X^{p}$ is $GV_{p-d_X -\delta_s(a)}$ for any $p\in \NN$.  
    \end{enumerate}
\end{proposition}

\begin{proof}
    The statement of $Ra_* \DB^{d_X}_X$ is a straightforward consequence of the classical generic vanishing theorem. Take a log resolution $f:\tilde X \to X$. Since $\tilde X$ is a smooth projective variety, we can apply generic vanishing theorem to the morphism $a \circ f$ to get
    \[Ra_* (Rf_* \omega_{\tilde X}) \cong R(a\circ f)_* \omega_{\tilde X}\]
    is $GV_{d_{a(X)}-d_X}$. By \cite[Proposition 3.3]{Steenbrink-vanishing}, we have 
    \[\DB^{d_X}_X \xto{\cong} Rf_* \omega_{\tilde X}\]
    which implies $Ra_* \DB^{d_X}_X$ is $GV_{d_{a(X)}-d_X}$.

    On the other hand, since $\IC_X\QQ^H$ is a polarizable Hodge module of pure weight $d_X$ with strict support $X$, the second statement is an application of Corollary \ref{corollary: GV of lowest Hodge filtration of pure HM}. 

    The generic Nakano vanishing follows from Corollary \ref{corollary: cohomological amplitude bound}, Corollary \ref{corollary: GV for D^bMHM}, Remark \ref{remark: semismallness upperbound for Q^H_X} and \ref{DB as grDR}, \ref{IC as grDR}.
\end{proof}

\begin{remark}
    It can be shown that $\DB_X^{d_X} \cong I\DB_X^{d_X} \cong f_* \omega_{\tilde X}$; thus two statements in $(1)$ are equivalent. 
\end{remark}

\medskip 

The other generalization $\DD_X(\DB^0_X)$ is more interesting in the sense that it will reflect the singularity of $X$. Let $X = \bigsqcup_{s\in sS} X_s$ be the minimal Whitney stratification of $X$. We define the \textit{stratified relative dimension} of the morphism $a$ as
\begin{equation} \label{definition: stratified relative dimension}
    d_F(a) \coloneqq \max\{\dim X_s - \dim a(X_s): s\in \sS \text{ and } X_s \subset X_{\sing}\}
\end{equation}

\begin{proposition} [Theorem \ref{theorem: hard GV for D(DB^p)}]
    Let $X$ be an irreducible projective variety and $a: X \to A$ a morphism to an abelian variety. Then 
    \begin{enumerate}
        \item $R a_* (\DD_X(\DB^0_X)[-d_X])$ is $GV_{d}$ where $d=\min\{d_{a(X)}-d_X, -d_F(a) -\lcdef(X)\}$. 
        \item $R a_* (\DD_X(\DB_X^{d_X-p})[-d_X])$ is $GV_{p-d_X-\delta_s(a)-\lcdef(X)}$.
    \end{enumerate}
\end{proposition}

\begin{proof} 
    The generic Nakano vanishing follows again from Corollary \ref{corollary: cohomological amplitude bound}, Corollary \ref{corollary: GV for D^bMHM} and \ref{DB as grDR}.
    
    For the first statement, let $M = \bD_X(\QQ^H_X[d_X])$. We know that $p(M) = 0$ and the statement can be written as $R a_* \DD_X(\DB^0_X)[-d_X] \cong R a_* \gr^F_{0}\DR(M) \cong \gr^F_0 \DR(a_* M)$ is $GV_{d}$, where the second isomorphism is due to Proposition \ref{proposition: grDR and f_* commute}. By Theorem \ref{theorem: cohom range of Q_X}, we have $H^i(M)=0$ unless $i$ is in the range between $0$ and $\lcdef(X)$. Because $IC_X\QQ^H[d_X](d_X) = \gr^W_{-d_X} H^0(M)$ and $M$ agree on $X_{\reg}$, we know that $H^i(M)$ (for $i>0$) and $\gr^W_{w} H^0(M)$ (for $w> -d_X$) are supported on the singular locus $X_{\sing}$. Moreover, by the strictness of the Hodge filtration and the compatibility with the weight filtration, we get
    \[p(\gr^W_w H^i M) \ge p(H^i M) \ge p(M) = 0\]
    for any $i, w$. 
    
    Consider the distinguished triangle of weight filtration
    \[W_{w-1} H^i(M) \to W_w H^i(M) \to \gr^W_w H^i(M) \xto{+1}\]
    for any $w$. Since $\gr^W_w H^i(M)$ is a polarizable Hodge modules of pure weight $w$, it decomposes into simple objects by strict support
    \[\gr^W_w H^i(M) = \bigoplus_{Z} M^i_Z.\]
    If either $i>0$ or $w > -d_X$, we have seen that $\gr^W_w H^i(M)$ is supported on $X_{\sing}$. Since $\text{rat}(M)$ restricts to local systems on $X_s$ for any $s\in \sS$, each $Z$ is a union of some strata $X_s$; so $d_{Z} - d_{a(Z)} \le d_F(a)$. Thus Corollary \ref{corollary: GV of lowest Hodge filtration of pure HM} implies that $\gr^F_0 \DR(a_* M^i_Z)$ is either $0$ if $p(M^i_Z) >0$ or $GV_{-d_F(a)}$ if otherwise; hence, we obtain
    \[\gr^F_{0} \DR(a_* \gr^W_w (H^iM))\]
    is $GV_{-d_F(a)}$ if $i>0$ or $w>-d_X$. 
    
    On the other hand, by \ref{theorem: easy GV for DB^p and IC^p}, we know
    \[\gr^F_{0} \DR(a_* \gr^W_{-d_X} (H^0M)) \cong Ra_* I\DB_X^{d_X}\]
    is $GV_{d_{a(X)}-d_X}$. By applying Corollary \ref{corollary: GV exact sequence}, we obtain
    \[\gr^F_0 \DR(a_*(H^0M)) \text{ is } GV_{\min\{d_{a(X)}-d_X, -d_F(a)\}}\]
    and 
    \[\gr^F_0 \DR(a_*(H^iM)) \text{ is } GV_{-d_F(a)} \text{ for } i>0.\]
    Putting these together and applying Corollary \ref{corollary: GV exact sequence} to the distinguished triangles
    \[\gr^F_{0} \DR( a_*(\tau^{< i}M)) \to \gr^F_{0} \DR(a_*(\tau^{<i+1}M)) \to \gr^F_{0} \DR(a_* H^{i}M)[-i] \xto{+1}\]
    for $0\le i \le \lcdef(X)$, the first statement follows. The proof is complete.  
\end{proof}

Let us consider some special cases in which the stratified relative dimension is easy to compute.

\begin{corollary} 
    Let $X$ be an irreducible projective variety with isolated singularities and $a: X \to A$ a morphism to an abelian variety. Then 
    \begin{enumerate}
        \item $R a_* \DD_X(\DB^0_X)[-d_X]$ is $GV_{d}$ where $d=\min\{d_{a(X)}-d_X, -\lcdef(X)\}$. 
        \item $R a_* \DD_X(\DB^{d_X-p}_X)[-d_X]$ is $GV_{p-d_X-\delta(a)-\lcdef(X)}$.
    \end{enumerate}
\end{corollary}

\begin{proof}
    When $X$ has isolated singularities, the minimal Whitney stratification is simply $X = X_{\reg} \bigsqcup X_{\sing}$ in which $\dim X_{\sing} = 0$. Thus $d_F(a) = 0$ and $\delta_s(a) = \delta(a)$. 
\end{proof}

In many moduli problems, we often encounter the situation of an equisingular family over a smooth base. A concrete formulation of this concept is Whitney equisingular morphism. 

\begin{definition} \cite{Sunggi-Viehweg-hyperbolicity}*{Definition 1.2}
    We say a morphism $a:X \to A$ is \textit{Whitney equisingular} if and only if $a|_{X_s}: X_s \to A$ is smooth for all $s\in \sS$, where $X= \bigsqcup_{s\in \sS} X_s$ is the minimal Whitney stratification of $X$. 
\end{definition}

Because smooth morphism is equidimensional, we know that all fibers of $a|_{X_s}:X_s \to A$ have the same dimension 
\[\dim X_s - \dim A\]
for any $y\in A, s\in \sS$. It is immediate that $d_F(a) = \dim X_{\sing} - \dim A$ and $\delta_s(a) = \delta(a) = \dim X - \dim A$. 

\begin{corollary}
    Let $X$ be an irreducible projective variety and $a: X \to A$ a Whitney equisingular morphism to an abelian variety. Then 
    \begin{enumerate}
        \item $R a_* \DD_X(\DB^0_X)[-d_X]$ is $GV_{d}$ where $d=d_A - \max\{d_X, d_{X_{\sing}}+\lcdef(X)\}$. 
        \item $R a_* \DD_X(\DB^{d_X-p}_X)[-d_X]$ is $GV_{p-2d_X+d_A-\lcdef(X)}$.
    \end{enumerate}
\end{corollary}

\medskip

\begin{example}
    Our result explains the counterexample from \cite[Theorem 1]{Hacon-counter_eg_for_GV}. Let us recall the set up of the example. Let $E_1, E_2$ be two cubic curves in $\PP^2$, and $H$ a general hyperplane section of $E_2$. Let $Y = C(E_1\times E_2)$ be the affine cone over $E_1\times E_2$ embedded in $\PP^8$ via the Veronese embedding, and $f:X=Bl_Z(Y) \to Y$ the blow of $Y$ along $Z = C(E_1\times H)$. The morphism $f$ is isomorphic away from the cone point $v$ and $f^{-1}(v) \cong E_1$. Moreover, $X$ has Gorenstein log canonical singularities with one dimensional singular locus $f^{-1}(v)$. Hacon and Kovács modify the morphism $f:X\to Y$ by a sequence of generically finite and \'etale maps to get a generically finite projective morphism $\lambda: T \to A$ to an abelian variety such that $T$ is isomorphic to $X$ over a neighborhood of $f^{-1}(v)$ and smooth outside that neighborhood. Hence $T$ also has Gorenstein log canonical singularities with one dimensional singular locus (see \cite[Proposition 3.13]{Hacon-counter_eg_for_GV}). One of the key features of this example is that $R^1f_* \omega_X \neq 0$, which is preserved under the above modification i.e. $R^1\lambda_* \omega_X \neq 0$. By \cite[Proposition 4.1]{Hacon-counter_eg_for_GV}, if the classical generic vanishing theorem holds for $T$, then $R^i \lambda_* \omega_T = 0$ for all $i > 0$; we get a contradiction. 
    
    In fact, we cannot expect the naive generic vanishing index without taking into account the singularities of $T$. Since $T$ is Gorenstein log canonical variety of dimension $3$, we have 
    \[\omega_T \cong \DD_T(\O_T)[-3] \cong \DD_T(\DB^0_T)[-3]\]
    and by \cite{dao-takagi_lcdef_bound}*{Corollary 2.8}, we have $\lcdef(T) = 0$. Moreover, the morphism $\lambda: T \to A$ maps the singular locus $T_{\sing}$ to a point; so $d_F(\lambda) = 1$. According to Theorem \ref{theorem: hard GV for D(DB^p)}, we get
    \[R\lambda_*\omega_T \text{ is } GV_{-1}.\]
    Thus $-1$ is the correct $GV$ index for Hacon and \Kovacs{}' counterexample. 
\end{example}

\begin{remark}
    \begin{enumerate}
        \item Our results recover the (dimensional) generic vanishing results in the smooth case, in which $\lcdef(X) =0, \delta_s(a) = \delta(a)$ and $d_F(a) = 0$. 
        \item If we look carefully into the proof of Proposition \ref{proposition: stratified semismall defect bound}, the intersection cohomology complex $\IC_X^\bullet$ is generated by $j_{s_!} \IC_X^\bullet|_{X_s}$. Moreover, by \cite[Lemma 3.3.11]{Achar-pervers-sheaves}, the local system $\IC_X^\bullet|_{X_s}$ is in $D_{locf}^b(X_s, \CC)^{\le -\dim X_s - 1}$, which gives us a slightly better bound 
        \[\max\{\delta(a|_{X_{\reg}}),\delta_s(a|_{X_{\sing}}) - 1\}\]
        for the cohomological amplitude of $a_* \IC_X\QQ^H$ in Corollary \ref{corollary: cohomological amplitude bound}.
        \item Because
        \[\DD_X(I\DB_X^{d_X-p})[-d_X] \cong I\DB_X^p\]
        by duality, they have similar $GV$ properties.
        \item The natural morphisms 
        \[\QQ^H_X[d_X] \to \IC_X\QQ^H \cong \bD_X(IC_X\QQ^H)(-d_X) \to \bD_X(\QQ^H_X[d_X])(-d_X)\]
        induce the canonical morphisms 
        \[\psi_p: \DB^p_X \to \DD_X(\DB^{d_X-p}_X)[-d_X]\]
        in $D_{coh}^b(X)$. If $\psi_p$ is an isomorphism, then we have a stronger result that  $\DD_X(\DB^{d_X-p}_X)[-d_X]$ is $GV_{p-d_X-\delta(a)}$. However, this imposes a strong condition on the singularities of $X$. If $\psi_p$ are isomorphism for all $p$, then 
        \[\H^{j-p+1} \DD_X (\DB^{d_X-p}_X) \cong \H^{j+d_X-p+1} \DB^p_X = 0\]
        for all $p$ and $j\ge 0$ by \cite[Theorem 7.29]{peters-steenbrink}. By Proposition \ref{proposition: lcd and DB}, we get $\lcdef(X) = 0$, which makes two $GV$ properties in Theorems \ref{theorem: easy GV for DB^p and IC^p}  and \ref{theorem: hard GV for D(DB^p)} coincide.  
    \end{enumerate}
\end{remark}

\bigskip

\subsection{Non-negativity of Euler characteristic of dual Du Bois complex}

Generic vanishing theorems are especially useful when the variety has maximal Albanese dimension i.e. $\dim X = \dim a(X)$, or equivalently, the Albanese morphism $a: X \to \Alb(X)$ is generically finite over its image. Examples of such varieties include subvarieties of abelian varieties, and their resolutions of singularities. One simple application of the classical generic vanishing theorem for smooth projective variety $X$ is that if $X$ has maximal Albanese dimension, then $\chi(X, \omega_X)\ge 0$. 

In the singular setting, we obtain an analogous statement by means of Theorem \ref{theorem: hard GV for D(DB^p)}. Thanks to \cite{FGA_explained_picard_scheme}*{Remark 9.5.25}, for a normal projective variety $X$, the Picard group $\Pic^0(X)$ and Albanese variety $\Alb(X) \cong \Pic^0(\Pic^0(X))$ exist as irreducible abelian varieties, and there exists the Albanese morphism $a: X\to \Alb(X)$, which satisfies the expected universal properties. Under this assumption, we have:   

\begin{proposition}[Proposition \ref{proposition: nonnegativity_Euler_characteristic}]
    Let $X$ be normal variety with maximal Albanese dimension such that $d_F(a) = 0$ and $\lcdef(X)=0$, then $\DD_X(\DB_X^0)[-d_X]$ is a $GV$-sheaf and $\chi(\DD_X(\DB_X^0)[-d_X])\ge 0$.  
\end{proposition}
\begin{proof}
    By Theorem \ref{theorem: hard GV for D(DB^p)}, we know that $\DD_X(\DB_X^0)[-d_X]$ is a $GV$-object. The assumption $\lcdef(X) = 0$ implies that 
    \[\H^i \DD_X(\DB_X^0)[-d_X] \cong \sExt_{\O_X}^{i-d_X}(\DB_X^0, \omega_X^\bullet) = 0\]
    for all $i > 0$, by Proposition \ref{proposition: lcd and DB}. On the other hand, thanks to \cite{DB_deform}*{Theorem 3.3}, we have 
    \[\H^i \DD_X(\DB_X^0)[-d_X] \cong \sExt_{\O_X}^{i-d_X}(\DB_X^0, \omega_X^\bullet) \hookrightarrow \H^{i-d_X} \omega_X^\bullet = 0\]
    for all $i < 0$. Thus $\DD_X(\DB_X^\bullet)[-d_X]$ is a sheaf. 

    Now, since $c_1(L) = 0$ for any $L\in \Pic^0(X)$, for a general line bundle $L\in \Pic^0(X)$ we have
    \begin{align*}
        \chi(\DD_X(\DB_X^0)[-d_X]) &= \chi(\DD_X(\DB_X^0)[-d_X]\otimes L) \quad (\text{for a general } L\in \Pic^0(X)) \\
        &= \sum_{i\ge 0} (-1)^i \dim H^i(\DD_X(\DB_X^0)[-d_X] \otimes L) \\
        &= \dim H^0(\DD_X(\DB_X^0)[-d_X] \otimes L) \quad (\text{because } \DD_X(\DB_X^0)[-d_X] \text{ is a $GV$-sheaf}) \\
        &\ge 0
    \end{align*}
\end{proof}

Let us analyze the extra assumptions $d_F(a) = 0$ and $\lcdef(X) = 0$:

\begin{example}
    \begin{enumerate}
        \item \textit{Varieties with $d_F(a) = 0$:} According to Definition \ref{definition: stratified relative dimension}, each strata $X_s \subset X_{\sing}$ is generically finite over its image via the Albanese morphism $a|_{X_s}:X_s\to A$; note that $a|_{X_s}$ is not necessarily the Albanese morphism of the strata $X_s$ itself as shown in the counterexmple from \cite{Hacon-counter_eg_for_GV}. Of course, subvarieties of abelian varieties are obvious examples. In addition, the condition $d_F(a) = 0$ holds when $X$ has isolated singularities, or $a:X \to A$ is a Whitney equisingular family of isolated singularities. 

        \item \textit{Varieties with $\lcdef(X) = 0$:} This condition is equivalent to $\lcd(X,Y) = \codim(X, Y)$ for any embedding of $X$ into a smooth variety $Y$. Of course, this holds when $X$ is a local complete intersection. Moreover, this also holds for arbitrary Cohen-Macaulay surfaces and threefolds due to \cite{Ogus-lcd}*{Remark p.338-339} and \cite{dao-takagi_lcdef_bound}*{Corollary 2.8}. We refer the readers to \cite{DB-injectivity}*{Example 4.4} for more examples. 
    \end{enumerate}
\end{example}

\bibliographystyle{alpha}
\bibliography{reference}

\end{document}